\documentclass[11pt]{amsart}

\usepackage{amssymb}
\usepackage{amsthm}
\usepackage{amsbsy}
\usepackage{subfigure}
\usepackage{amsmath}
\usepackage{graphicx}

\usepackage{comment}
\usepackage[usenames,dvipsnames]{xcolor}
\usepackage[margin = 1.1in]{geometry}
\usepackage{enumerate}
\usepackage[toc,page]{appendix}
\usepackage{lineno}
\usepackage{color}
\usepackage{hyperref}
\usepackage[ruled,vlined,linesnumbered]{algorithm2e}

\SetKwInOut{Parameter}{Parameters}

\definecolor{mygreen}{rgb}{0.1,0.75,0.2}

 \newtheorem{thm}{Theorem}[section]

 \newtheorem{prop}[thm]{Proposition}

\numberwithin{equation}{section}

\DeclareMathOperator{\sgn}{sgn}

\newcommand{\Lra}{\Longrightarrow}

\providecommand{\bbs}[1]{\left(#1\right)}

\newcommand{\hs}{\mathcal{H}}

\newcommand{\bR}{\mathbb{R}}

\newcommand{\la}{\langle}
\newcommand{\ra}{\rangle}
\newcommand{\pt}{\partial}

\newcommand{\ud}{\,\mathrm{d}}
\newcommand{\8}{\infty}

\newcommand{\nn}{\mathcal{N}}
\newcommand{\mx}{\mathbf{x}}
\newcommand{\my}{\mathbf{y}}

\newcommand{\divn}{\text{div}_{\nn}}

\newcommand{\iright}{\frac{\pi_{i,j}+ \pi_{i+1,j}}{2\Delta x^2 \pi_{i,j}} }
\newcommand{\ileft}{\frac{\pi_{i,j}+ \pi_{i-1,j}}{2\Delta x^2 \pi_{i,j}} }
\newcommand{\jup}{\frac{\pi_{i,j+1}+ \pi_{i,j}}{2\Delta y^2 \pi_{i,j}} }
\newcommand{\jdown}{\frac{\pi_{i,j}+ \pi_{i,j-1}}{2\Delta y^2 \pi_{i,j}} }

\newcommand{\irightn}{\frac{\pi_{i,j}+ \pi_{i+1,j}}{2\Delta x^2} }
\newcommand{\ileftn}{\frac{\pi_{i,j}+ \pi_{i-1,j}}{2\Delta x^2} }
\newcommand{\jupn}{\frac{\pi_{i,j+1}+ \pi_{i,j}}{2\Delta y^2 } }
\newcommand{\jdownn}{\frac{\pi_{i,j}+ \pi_{i,j-1}}{2\Delta y^2 } }

\begin{document}

\title{Inbetweening auto-animation via Fokker-Planck dynamics and thresholding}

\author[Y. Gao]{Yuan Gao}
\address{Department of Mathematics, Duke University, Durham, NC}
\email{yuangao@math.duke.edu}

\author[G. Jin]{Guangzhen Jin}
\address{Southern marine science and engineering Guangdong laboratory, Zhuhai, China\\
Key laboratory of marine resources and coastal engineering in Guangdong province, school of marine sciences, Sunyat‐sen University, Guangzhou, China}
\email{jingzh3@mail.sysu.edu.cn; gzjinouc@gmail.com}
\author[J.-G. Liu]{Jian-Guo Liu}
\address{Department of Mathematics and Department of
  Physics, Duke University, Durham, NC}
\email{jliu@math.duke.edu}

\begin{abstract}
We propose an equilibrium-driven deformation algorithm (EDDA) to simulate the inbetweening transformations starting from
 an initial image to an equilibrium image, which covers images varying from a greyscale type to a colorful type on plane or manifold.  The algorithm is based on Fokker-Planck dynamics on manifold,
which  automatically cooperates positivity, unconditional stability, mass conservation law, exponentially convergence and also  the manifold structure suggested by dataset.
The thresholding scheme is adapted for the sharp interface dynamics and is used to achieve the finite time convergence.
Using EDDA, three challenging examples, (I) facial aging process, (II) coronavirus disease 2019 (COVID-19) invading/treatment process, and (III) continental evolution process are conducted efficiently. 
\end{abstract}

\date{\today}
\maketitle
\section{Introduction}
Inbetweening auto-animation is 
to automatically generate animations (motions) given a starting and end images. The classical method for auto-animation use detailed kinematic equations for each object in the starting images, which is precise but time consuming due to case by case c.f. \cite{cong2015fully, zollhofer2018state}. 

Instead of analyzing the detailed kinematic equation for each object,
we aim to propose an efficient and universal algorithm for inbetweening auto-animation based on Fokker-Planck dynamics on manifold and thresholding. We call this algorithm equilibrium-driven deformation algorithm (EDDA).

EDDA regards the end image as an equilibrium state of a Fokker-Planck equation  and the inbetweening motion is driven by an underlying potential force determined by the equilibrium. This viewpoint is especially useful when the detailed physical process is not clear or hard to describe. For instance, the inbetweening motion of aging process, tumor growth,  pneumonia invading for coronavirus disease 2019 (COVID-19)  or the formation of current continents/oceans starting from pangaea.

We first consider a  Fokker-Planck equation in a flat domain $\Omega \subset \bR^\ell$ with a unique equilibrium $\pi$ and no-flux boundary condition in Section \ref{sec2} and then we propose an efficient solver for this Fokker-Planck equation in Section \ref{sec3}. The numerical solver for this part is based on structured grids and finite volume method \cite{eymard2000finite}. An unconditionally stable explicit time discretization is introduced, which automatically enjoys positivity, mass conservation law, exponentially convergence and also  efficiency.
For a Fokker-Planck equation on a closed manifold, we propose a similar efficient solver based on point clouds and the associated Voronoi tessellation in Section \ref{sec_fp_n}.  The Voronoi tessellation automatically gives the manifold information and can be used to approximate surface gradient/divergence in the Fokker-Planck equation. Based on this,  an analogue unconditionally stable explicit time discretization is introduced.

To realize the end image (the equilibrium) at finite time and the sharp dynamics of the inbetweening motion, we combine the explicit-time-discretization of the  Fokker-Planck equation with 
the thresholding dynamics. When the equilibrium image has a sharp interface, the scheme adapting thresholding step converges faster than the purely  Fokker-Planck iteration and the relative error reaches machine accuracy at finite time.

In Section \ref{sec_simu}, we apply EDDA proposed for  either structured grids on $\Omega\subset\bR^\ell$, or for point-clouds which suggests an underlining manifold to conduct three challenging and important examples: (I) facial aging process, (II) COVID-19 invading/treatment process, and (III) continental evolution process. In Example (I), inbetweening facial aging process at each time is simulated and potentially reveals the detailed changes of different part of human face over time. In Example (II), the inbetweening evolution of COVID-19 pneumonia invading before treatment and the fading away after treatment are simulated, which shows a good agreement with computerized tomography (CT)  scans and also reveals promising application in the studying of pathology  for COVID-19. In Example (III), the Fokker-Planck dynamics and thresholding are combined together to simulate the continental drifting process, which may suggest a new explanation for the formation of the current five continents of the world. From those examples from quite different research fields, EDDA are shown to be a very efficient and universal method with enormous potential applications in other fields of science and  industry.

\section{Fokker-Planck equation and equilibrium}\label{sec2}
Suppose $\Omega\subset \bR^\ell$ is a closed subset in $\bR^\ell$.   Assume the end image on $\Omega$ is described by a equilibrium density $\pi(\mx): \Omega \to \bR.$
The value of $\rho$ indicates the gray level of the image for a grayscale image. In the case of Red-Green-Blue (RGB)  image, we use three separate densities to indicate the RGB levels of the image separately.
Then with $\pi \propto e^{-\phi}$, the  Fokker-Planck equation is given by
\begin{equation}\label{FPo}
\pt_t \rho = \Delta \rho + \nabla \cdot(\rho \nabla \phi)= \nabla \cdot \bbs{\pi \nabla \bbs{\frac{\rho}{\pi}}}
\end{equation}
with initial data $\rho_0$ satisfying 
\begin{equation}
\int_{\Omega} \rho_0 \ud x = \int_{\Omega} \pi \ud x.
\end{equation}
We consider the following natural no-flux boundary condition
\begin{equation}\label{bc}
n\cdot  \nabla \bbs{\frac{\rho}{\pi}} =0 \quad \text{on }\pt \Omega.
\end{equation}
\eqref{FPo} can be recast as the relative entropy formulation
\begin{equation}
\pt_t \rho = \nabla \cdot \bbs{ \rho \nabla \ln\frac{\rho}{\pi}} .
\end{equation}

See Section \ref{sec_fp_n} for a Fokker-Planck equation  on a $d$ dimensional smooth closed submanifold of $\mathbb{R}^\ell$.

Now we state the ergodicity of Fokker-Planck equation \eqref{FPo}. Assume 
\begin{equation}
\pi>0, \quad \pi\in C^1(\bar{\Omega}).
\end{equation}
Let $L^2(\Omega;\frac{1}{\pi}\ud x)$ be the weighted Sobolev space. Define the  Fokker-Planck operator for \eqref{FPo} as $L^*: D(L^*)\subset L^2(\Omega; \frac{1}{\pi} \ud x) \to \bR$ with $D(L^*):=\{u\in H^2(\Omega ;\frac{1}{\pi}\ud x); \pt_n \frac{u}{\pi} =0 \text{ on }\pt \Omega\}$
\begin{equation}
L^* u:= -\nabla \cdot (\pi \nabla \frac{u}{\pi}).
\end{equation}
This can be regarded as the adjoint operator of the generator $L=-\frac{1}{\pi} \nabla \cdot(\pi \nabla)$ of Fokker-Planck equation \eqref{FPo}.
One can see $L^*$ is self-adjoint operator in $L^2(\Omega; \frac{1}{\pi} \ud x)$ with compact resolvent $(\lambda I +L^*)^{-1}$ for $\lambda$ large enough. Thus $L^*$ has only discrete spectrum without finite accumulation points. Furthermore, since $\pi>0$, for $\rho\in D(L)$,
\begin{equation}
L^*\rho = 0, \,\,\, \Lra \int \pi |\nabla \frac{\rho}{\pi}|^2 \ud x =0, \,\,\, \Lra \rho =c \pi.
\end{equation}  
Therefore, we conclude $0$ is the simple principal eigenvalue of $L^*$  with the corresponding eigenfunction $\pi$, which leads to the spectral gap of $L^*$ in $L^2(\Omega; \frac{1}{\pi} \ud x)$, i.e.
\begin{equation}
\la L^* u, u \ra_{\frac{1}{\pi}} \geq c \la u, u\ra_{\frac{1}{\pi}}, \quad \text{ for }u \text{ s.t. } \la u, \pi \ra_{\frac{1}{\pi}}=0.
\end{equation}
Thus due to $\int (\rho-\pi) \ud x=0$, we have the following Poincare's inequality
\begin{equation}
\int |\nabla \bbs{\frac{\rho}{\pi} -1}|^2 \pi \ud x \geq c \int \bbs{\frac{\rho}{\pi}-1}^2 \pi \ud x.
\end{equation}
Therefore,  multiplying \eqref{FPo} by $\frac{\rho}{\pi}-1$, by \eqref{bc} we have
\begin{equation}
\frac{1}{2} \frac{\ud }{\ud t} \int \frac{(\rho-\pi)^2}{\pi} \ud x = -\int \pi |\nabla \frac{\rho}{\pi}|^2 \ud x \leq -  c \int \frac{(\rho-\pi)^2}{\pi} \ud x,
\end{equation}
which gives the ergodicity that 
\begin{equation}
\|\rho(\cdot, t)-\pi\|_{L^2(\Omega; \frac{1}{\pi} \ud x)} \leq 
 e^{-2ct} \|\rho(\cdot, 0)-\pi\|_{L^2(\Omega; \frac{1}{\pi} \ud x)} .
\end{equation}

\section{EDDA based on structured grids}\label{sec3}

We present the  numerical method based on structured grids for a Fokker-Planck equation on 2D domain $\Omega:= [a,b]\times [c,d]$. Let the grid size be $\Delta x= \frac{b-a}{N}, \, \Delta y = \frac{d-c}{M}$. Define the cells as
\begin{equation}
C_{ij} = ((i-1)\Delta x, i \Delta x) \times ((j-1)\Delta y, j\Delta y), \quad i=1, \cdots, N, \,\, j=1, \cdots, M.
\end{equation}
Then the cell centers $(x_i,y_i)$ are
\begin{equation}
x_i = a+ (i-\frac12)\Delta x, \quad y_j= c+ (j-\frac12) \Delta y,  \quad i=1, \cdots, N, \,\, j=1, \cdots, M.
\end{equation}
We use $\rho_{i,j}$ to approximate the value of $\rho(x_i, y_j)$ and take $\pi_{i,j}=\pi(x_i, y_j).$ Then the continuous-time finite volume scheme
is
\begin{equation}\label{fp2d}
\begin{aligned}
\dot{\rho}_{i,j} =& \frac{1}{\Delta x ^2} \bbs{ \frac{\pi_{i,j}+ \pi_{i+1,j}}{2}\bbs{\frac{\rho_{i+1,j}}{\pi_{i+1,j}} - \frac{\rho_{i,j}}{\pi_{i,j}} } - \frac{\pi_{i-1,j} + \pi_{i,j}}{2} \bbs{  \frac{\rho_{i,j}}{\pi_{i,j}} - \frac{\rho_{i-1,j}}{\pi_{i-1,j}} }} \\
&+ \frac{1}{\Delta y ^2} \bbs{ \frac{\pi_{i,j}+ \pi_{i,j+1}}{2}\bbs{\frac{\rho_{i,j+1}}{\pi_{i,j+1}} - \frac{\rho_{i,j}}{\pi_{i,j}} } - \frac{\pi_{i,j-1} + \pi_{i,j}}{2} \bbs{  \frac{\rho_{i,j}}{\pi_{i,j}} - \frac{\rho_{i,j-1}}{\pi_{i,j-1}} }} 
\end{aligned}
\end{equation}
for $i=1, \cdots, N, \, j=1, \cdots, M$ with the no-flux boundary condition \eqref{bc}. We assume the equilibrium $\pi$ satisfies 
\begin{equation}\label{nbcpi}
\begin{aligned}
\pi_{0,j} = \pi_{1,j},\, \pi_{N+1, j}= \pi_{N,j} \quad j=1, \cdots, M,\\
\pi_{i,0} = \pi_{i,1},\, \pi_{i,M+1}= \pi_{i,M} \quad i=1, \cdots, N,\\
\end{aligned}
\end{equation}
then the no-flux boundary condition \eqref{bc} is reduced to
\begin{equation}\label{nbc}
\begin{aligned}
\rho_{0,j} = \rho_{1,j},\, \rho_{N+1, j}= \rho_{N,j} \quad j=1, \cdots, M,\\
\rho_{i,0} = \rho_{i,1},\, \rho_{i,M+1}= \rho_{i,M} \quad i=1, \cdots, N.\\
\end{aligned}
\end{equation}

Denote $\rho^k_{i,j}$ as the value of $\rho$ at $t^k=k\Delta t$ with time step size $\Delta t.$ Now we introduce an unconditionally stable explicit time discretization for \eqref{fp2d}
\begin{equation}\label{explicit_scheme}
\begin{aligned}
\frac{\rho^{k+1}_{i,j}-\rho^{k}_{i,j}}{\Delta t} =& \frac{1}{\Delta x ^2} \bbs{ \frac{\pi_{i,j}+ \pi_{i+1,j}}{2}\bbs{\frac{\rho^k_{i+1,j}}{\pi_{i+1,j}} - \frac{\rho^{k+1}_{i,j}}{\pi_{i,j}} } - \frac{\pi_{i-1,j} + \pi_{i,j}}{2} \bbs{  \frac{\rho^{k+1}_{i,j}}{\pi_{i,j}} - \frac{\rho^k_{i-1,j}}{\pi_{i-1,j}} }} \\
&+ \frac{1}{\Delta y ^2} \bbs{ \frac{\pi_{i,j}+ \pi_{i,j+1}}{2}\bbs{\frac{\rho^{k}_{i,j+1}}{\pi_{i,j+1}} - \frac{\rho^{k+1}_{i,j}}{\pi_{i,j}} } - \frac{\pi_{i,j-1} + \pi_{i,j}}{2} \bbs{  \frac{\rho^{k+1}_{i,j}}{\pi_{i,j}} - \frac{\rho^k_{i,j-1}}{\pi_{i,j-1}} }}
\end{aligned}
\end{equation}
for $i=1, \cdots, N, \, j=1, \cdots, M$ with the no-flux boundary condition \eqref{nbc}.

We now further simplify \eqref{explicit_scheme} as
\begin{equation}\label{tm38}
\begin{aligned}
&\bbs{1+\frac{{\Delta t}}{\Delta x ^2} \bbs{\frac{\pi_{i,j}+ \pi_{i+1,j}}{2\pi_{i,j}} + \frac{\pi_{i-1,j} + \pi_{i,j}}{2\pi_{i,j}} } + \frac{{\Delta t}}{\Delta y ^2}  \bbs{ \frac{\pi_{i,j}+ \pi_{i,j+1}}{2\pi_{i,j}} +  \frac{\pi_{i,j-1} + \pi_{i,j}}{2\pi_{i,j}}  } }\frac{\rho^{k+1}_{i,j}}{\pi_{i,j}}\\
=
&\frac{\rho^{k}_{i,j}}{\pi_{i,j}} + \frac{{\Delta t}}{\Delta x ^2} \bbs{ \frac{\pi_{i,j}+ \pi_{i+1,j}}{2\pi_{i,j}} \frac{\rho^k_{i+1,j}}{\pi_{i+1,j}}   + \frac{\pi_{i-1,j} + \pi_{i,j}}{2\pi_{i,j}}      \frac{\rho^k_{i-1,j}}{\pi_{i-1,j}} } \\
&\quad+ \frac{{\Delta t}}{\Delta y ^2} \bbs{ \frac{\pi_{i,j}+ \pi_{i,j+1}}{2\pi_{i,j}}\frac{\rho^{k}_{i,j+1}}{\pi_{i,j+1}}   + \frac{\pi_{i,j-1} + \pi_{i,j}}{2\pi_{i,j}}   \frac{\rho^k_{i,j-1}}{\pi_{i,j-1}} }.
\end{aligned}
\end{equation}
Define
\begin{equation}
\lambda_{i,j}:= \frac{1}{\Delta x ^2} \bbs{\frac{\pi_{i,j}+ \pi_{i+1,j}}{2\pi_{i,j}} + \frac{\pi_{i-1,j} + \pi_{i,j}}{2\pi_{i,j}} } + \frac{1}{\Delta y ^2}  \bbs{ \frac{\pi_{i,j}+ \pi_{i,j+1}}{2\pi_{i,j}} +  \frac{\pi_{i,j-1} + \pi_{i,j}}{2\pi_{i,j}}  }.
\end{equation}
Then \eqref{tm38} can be rewritten as
\begin{equation}\label{tm310}
\begin{aligned}
(1+\Delta t \lambda_{i,j}) \rho_{i,j}^{k+1} =&  \rho^{k}_{i,j} + \frac{{\Delta t}}{\Delta x ^2} \bbs{ \frac{\pi_{i,j}+ \pi_{i+1,j}}{2} \frac{\rho^k_{i+1,j}}{\pi_{i+1,j}}   + \frac{\pi_{i-1,j} + \pi_{i,j}}{2}      \frac{\rho^k_{i-1,j}}{\pi_{i-1,j}} } \\
\quad&+ \frac{{\Delta t}}{\Delta y ^2} \bbs{ \frac{\pi_{i,j}+ \pi_{i,j+1}}{2}\frac{\rho^{k}_{i,j+1}}{\pi_{i,j+1}}   + \frac{\pi_{i,j-1} + \pi_{i,j}}{2}   \frac{\rho^k_{i,j-1}}{\pi_{i,j-1}} }.
\end{aligned}
\end{equation}
Denote $h:=\max\{\Delta x, \Delta y\}$. From \eqref{tm310}, we recast the scheme using a rescaled generator operator
\begin{equation}\label{Dgen}
\begin{aligned}
 \frac{\rho^{k+1}_{i,j} }{\pi_{i,j} }- \frac{\rho^{k}_{i,j} }{\pi_{i,j} } =&\frac{1}{1+\Delta t \lambda_{i,j}} \frac{{\Delta t}}{\Delta x ^2} \bbs{ \frac{\pi_{i,j}+ \pi_{i+1,j}}{2\pi_{i,j}} \frac{\rho^k_{i+1,j}}{\pi_{i+1,j}}   + \frac{\pi_{i-1,j} + \pi_{i,j}}{2\pi_{i,j}}      \frac{\rho^k_{i-1,j}}{\pi_{i-1,j}} } \\
&\quad+ \frac{1}{1+\Delta t \lambda_{i,j}} \frac{{\Delta t}}{\Delta y ^2} \bbs{ \frac{\pi_{i,j}+ \pi_{i,j+1}}{2\pi_{i,j}}\frac{\rho^{k}_{i,j+1}}{\pi_{i,j+1}}   + \frac{\pi_{i,j-1} + \pi_{i,j}}{2\pi_{i,j}}   \frac{\rho^k_{i,j-1}}{\pi_{i,j-1}} } 
- \frac{\Delta t \lambda_{i,j}}{1+\Delta t \lambda_{i,j}}\frac{\rho^{k}_{i,j} }{\pi_{i,j}}\\
=& \frac{\Delta t}{1+\Delta t \lambda_{i,j}} \Big[ \iright \bbs{\frac{\rho^k_{i+1,j}}{\pi_{i+1,j}}  - \frac{\rho^{k}_{i,j} }{\pi_{i,j}} } -  \ileft \bbs{\frac{\rho^k_{i,j}}{\pi_{i,j}}  - \frac{\rho^{k}_{i-1,j} }{\pi_{i-1,j}} } \\
&+ \jup \bbs{\frac{\rho^k_{i,j+1}}{\pi_{i,j+1}}  - \frac{\rho^{k}_{i,j} }{\pi_{i,j}} }  - \jdown \bbs{ \frac{\rho^k_{i,j}}{\pi_{i,j}}  - \frac{\rho^{k}_{i,j-1} }{\pi_{i,j-1}}  }    \Big]=: -\Delta t (L_h \frac{\rho^k}{\pi})_{i,j},
\end{aligned}
\end{equation}

Now we state the the positivity, maximal principle, mass conservation law and ergodicity of the scheme \eqref{tm38} as follows.
\begin{prop}\label{prop1}
Let $\pi_{i,j}=\pi(x_i,y_j)>0$. Let $\Delta t$ be the time step and consider the explicit scheme \eqref{explicit_scheme} for the numerical solution $\rho^k_{i,j}$ with \eqref{nbc}. Assume the initial data $\rho^0>0$ satisfies 
\begin{equation}\label{i_adj}
\sum_{i=1}^{N} \sum_{j=1}^{M} (1+\Delta t \lambda_{i,j}) \rho_{i,j}^{0} =  \sum_{i=1}^{N} \sum_{j=1}^{M}  (1+\Delta t \lambda_{i,j}) \pi_{i,j}.
\end{equation}
Then
we have
\begin{enumerate}[(i)]
\item positivity preserving property
\begin{equation}
\rho^k_{i,j}>0,\,\, i=1, \cdots, N, \, j=1, \cdots, M \quad \Lra \quad \rho^{k+1}_{i,j}>0,\,\, i=1, \cdots, N, \, j=1, \cdots, M;
\end{equation}
\item the mass-conversation law 
\begin{equation}\label{conser1_n}
\sum_{i=1}^{N} \sum_{j=1}^{M} (1+\Delta t \lambda_{i,j}) \rho_{i,j}^{k+1} =  \sum_{i=1}^{N} \sum_{j=1}^{M}  (1+\Delta t \lambda_{i,j}) \rho_{i,j}^{k}.
\end{equation}
\item  the unconditional maximal principle for $\frac{\rho_{i,j}}{\pi_{i,j}}$
\begin{equation}\label{maxP_n}
\max_{i,j} \frac{\rho^{k+1}_{i,j} }{\pi_{i,j}}\leq \max_{i,j} \frac{\rho^{k}_{i,j} }{\pi_{i,j}};
\end{equation}
\item the $\ell^\8$ contraction
\begin{equation}\label{l8semi_n}
\max_{i,j} \left| \frac{\rho^{k+1}_{i,j}}{\pi_{i,j}}-1\right| \leq \max_{i,j} \left| \frac{\rho^{k}_{i,j}}{\pi_{i,j}}-1\right| ;
\end{equation}
\item the exponential convergence
\begin{equation}\label{gap_error_n}
\left\| \frac{\rho^{k}_{i,j}}{\pi_{i,j}}-1\right\|_{\ell_F} \leq c  |\mu_2|^k, \quad |\mu_2|<1,
\end{equation}
where $\mu_2$ is the second eigenvalue of $A$ defined in \eqref{gap26}. 
\end{enumerate}
\end{prop}
\begin{proof}
For (i), from \eqref{tm310}, since $\pi>0$, we know $\rho^k_{i,j}>0$ implies $\rho^{k+1}_{i,j}>0$.

To prove (ii), taking summation in \eqref{tm310}, we have the 
\begin{equation}\label{massb1}
\begin{aligned}
\sum_{i=1}^{N} \sum_{j=1}^{M}(1+\Delta t \lambda_{i,j}) \rho_{i,j}^{k+1} =& \sum_{i=1}^{N} \sum_{j=1}^{M} \rho^{k}_{i,j} + \sum_{i=1}^{N} \sum_{j=1}^{M} \frac{{\Delta t}}{\Delta x ^2} \bbs{ \frac{\pi_{i,j}+ \pi_{i+1,j}}{2} \frac{\rho^k_{i+1,j}}{\pi_{i+1,j}}   + \frac{\pi_{i-1,j} + \pi_{i,j}}{2}      \frac{\rho^k_{i-1,j}}{\pi_{i-1,j}} } \\
\quad&+ \sum_{i=1}^{N} \sum_{j=1}^{M} \frac{{\Delta t}}{\Delta y ^2} \bbs{ \frac{\pi_{i,j}+ \pi_{i,j+1}}{2}\frac{\rho^{k}_{i,j+1}}{\pi_{i,j+1}}   + \frac{\pi_{i,j-1} + \pi_{i,j}}{2}   \frac{\rho^k_{i,j-1}}{\pi_{i,j-1}} }.
\end{aligned}
\end{equation}
The second term in the RHS of \eqref{massb1}  is
\begin{equation}
\begin{aligned}
&\sum_{i=1}^{N} \sum_{j=1}^{M} \frac{{\Delta t}}{\Delta x ^2} \frac{\pi_{i,j}+ \pi_{i+1,j}}{2} \frac{\rho^k_{i+1,j}}{\pi_{i+1,j}}     =  \sum_{i=2}^{N+1} \sum_{j=1}^{M} \frac{{\Delta t}}{\Delta x ^2} \frac{\pi_{i-1,j}+ \pi_{i,j}}{2} \frac{\rho^k_{i,j}}{\pi_{i,j}}  \\
=&  \sum_{i=1}^{N} \sum_{j=1}^{M} \frac{{\Delta t}}{\Delta x ^2} \frac{\pi_{i-1,j}+ \pi_{i,j}}{2} \frac{\rho^k_{i,j}}{\pi_{i,j}} +   \frac{{\Delta t}}{\Delta x ^2} \frac{\pi_{N,j}+ \pi_{N+1,j}}{2} \frac{\rho^k_{N+1,j}}{\pi_{N+1,j}}  - \frac{{\Delta t}}{\Delta x ^2} \frac{\pi_{0,j}+ \pi_{1,j}}{2} \frac{\rho^k_{1,j}}{\pi_{1,j}} \\
=&  \sum_{i=1}^{N} \sum_{j=1}^{M} \frac{{\Delta t}}{\Delta x ^2} \frac{\pi_{i-1,j}+ \pi_{i,j}}{2} \frac{\rho^k_{i,j}}{\pi_{i,j}} +   \frac{{\Delta t}}{\Delta x ^2} \bbs{\rho^k_{N+1,j}-\rho^k_{1,j}},
\end{aligned}
\end{equation}
where we used the no-flux boundary condition \eqref{nbcpi}.
Similarly, the third term in the RHS of \eqref{massb1} is
\begin{equation}
\begin{aligned}
&\sum_{i=1}^{N} \sum_{j=1}^{M} \frac{{\Delta t}}{\Delta x ^2} \frac{\pi_{i-1,j}+ \pi_{i,j}}{2} \frac{\rho^k_{i-1,j}}{\pi_{i-1,j}}     =  \sum_{i=0}^{N-1} \sum_{j=1}^{M} \frac{{\Delta t}}{\Delta x ^2} \frac{\pi_{i+1,j}+ \pi_{i,j}}{2} \frac{\rho^k_{i,j}}{\pi_{i,j}}  \\
=&  \sum_{i=1}^{N} \sum_{j=1}^{M} \frac{{\Delta t}}{\Delta x ^2} \frac{\pi_{i+1,j}+ \pi_{i,j}}{2} \frac{\rho^k_{i,j}}{\pi_{i,j}} -   \frac{{\Delta t}}{\Delta x ^2} \frac{\pi_{N,j}+ \pi_{N+1,j}}{2} \frac{\rho^k_{N,j}}{\pi_{N,j}}  + \frac{{\Delta t}}{\Delta x ^2} \frac{\pi_{0,j}+ \pi_{1,j}}{2} \frac{\rho^k_{0,j}}{\pi_{0,j}} \\
=&  \sum_{i=1}^{N} \sum_{j=1}^{M} \frac{{\Delta t}}{\Delta x ^2} \frac{\pi_{i+1,j}+ \pi_{i,j}}{2} \frac{\rho^k_{i,j}}{\pi_{i,j}} -   \frac{{\Delta t}}{\Delta x ^2} \bbs{\rho^k_{N,j}-\rho^k_{0,j}} .
\end{aligned}
\end{equation}
One can shift index for the last two terms in the RHS of \eqref{massb1} similarly.  Therefore, using the no-flux boundary condition \eqref{nbc}, we have the mass balance
\begin{equation}
\sum_{i=1}^{N} \sum_{j=1}^{M} (1+\Delta t \lambda_{i,j}) \rho_{i,j}^{k+1} =  \sum_{i=1}^{N} \sum_{j=1}^{M}  (1+\Delta t \lambda_{i,j}) \rho_{i,j}^{k}.
\end{equation}

To prove (iii), directly taking maximum in the RHS of \eqref{tm38} implies
\begin{equation}
(1+\Delta t \lambda_{i,j})\frac{\rho^{k+1}_{i,j} }{\pi_{i,j}}\leq (1+\Delta t \lambda_{i,j})\max_{i,j} \frac{\rho^{k}_{i,j} }{\pi_{i,j}},
\end{equation}
which leads to \eqref{maxP}. 

To prove (iv), subtract $(1+\Delta t\lambda_{i,j})$ from both sides of \eqref{tm38} and then multiply by $\sgn\bbs{\frac{\rho^{k+1}_{i,j} }{\pi_{i,j}}-1}$. Thus using same argument with (iii), we have
\begin{equation}
(1+\Delta t \lambda_{i,j})\left|\frac{\rho^{k+1}_{i,j} }{\pi_{i,j}}-1\right| \leq (1+\Delta t \lambda_{i,j})\max_{i,j} \left|\frac{\rho^{k}_{i,j} }{\pi_{i,j}}-1\right|,
\end{equation}
which implies \eqref{l8semi}.

Now we  prove (v).  Recall \eqref{Dgen}, i.e.
\begin{equation}
\begin{aligned}
 \frac{\rho^{k+1}_{i,j} }{\pi_{i,j} }- \frac{\rho^{k}_{i,j} }{\pi_{i,j} } =&  -\Delta t (L_h \frac{\rho^k}{\pi})_{i,j} \\
=& \frac{\Delta t}{1+\Delta t \lambda_{i,j}} \Big[ \iright \bbs{\frac{\rho^k_{i+1,j}}{\pi_{i+1,j}}  - \frac{\rho^{k}_{i,j} }{\pi_{i,j}} } -  \ileft \bbs{\frac{\rho^k_{i,j}}{\pi_{i,j}}  - \frac{\rho^{k}_{i-1,j} }{\pi_{i-1,j}} } \\
&+ \jup \bbs{\frac{\rho^k_{i,j+1}}{\pi_{i,j+1}}  - \frac{\rho^{k}_{i,j} }{\pi_{i,j}} }  - \jdown \bbs{ \frac{\rho^k_{i,j}}{\pi_{i,j}}  - \frac{\rho^{k}_{i,j-1} }{\pi_{i,j-1}}  }    \Big].
\end{aligned}
\end{equation}

 By shifting index and no-flux boundary condition \eqref{nbc} we have
\begin{equation}
\begin{aligned}
\la -L_h \frac{\rho^k}{\pi}, \rho^k\ra_{1+\Delta t \lambda}=&-\sum_{i,j}\Big[ \irightn \bbs{\frac{\rho^k_{i+1,j}}{\pi_{i+1,j}}  - \frac{\rho^{k}_{i,j} }{\pi_{i,j}} }^2 +  \ileftn \bbs{\frac{\rho^k_{i,j}}{\pi_{i,j}}  - \frac{\rho^{k}_{i-1,j} }{\pi_{i-1,j}} }^2 \\
&+ \jupn \bbs{\frac{\rho^k_{i,j+1}}{\pi_{i,j+1}}  - \frac{\rho^{k}_{i,j} }{\pi_{i,j}} }^2  + \jdownn \bbs{ \frac{\rho^k_{i,j}}{\pi_{i,j}}  - \frac{\rho^{k}_{i,j-1} }{\pi_{i,j-1}}  }^2    \Big].
\end{aligned}
\end{equation}
From this, one know 
\begin{equation}
L_h \frac{\rho^\8}{\pi} = 0, \quad \Lra \quad  \la L_h \frac{\rho^\8}{\pi}, \rho^\8\ra_{{1+\Delta t \lambda}} =0\quad \Lra \quad \rho^\8 = c \pi = \pi.
\end{equation}

Denote $u^k_{i,j}=\frac{\rho^k_{i,j}}{\pi_{i,j}}$. Then \eqref{Dgen} is recast as
\begin{equation}\label{gap26}
\begin{aligned}
u_{i,j}^{k+1} =& \frac{1}{1+\Delta t \lambda_{i,j}} \Big[ u^{k}_{i,j} + \Delta t \bbs{\iright u^k_{i+1,j} + \ileft u^k_{i-1,j}}\\
\quad&+\Delta t \jup u^{k}_{i,j+1} + \jdown u^k_{i,j-1}\Big]=: (Au^k)_{i,j}.
\end{aligned}
\end{equation}
By the Perron-Frobenius theorem, $\mu_1=1$ is the 
simple, principal eigenvalue of $A$ with the ground state $u^*_{i,j}\equiv 1$ and other eigenvalues $\mu_i$ of $A$ satisfy $|\mu_i|<\mu_1$. Notice also the mass conservation for initial data $u^0=\frac{\rho^0}{\pi}$ satisfying \eqref{conser1}, i.e.,
\begin{equation}
\la u^0 - u^*, u^* \ra_{(1+\Delta t \lambda)\pi} =0.
\end{equation}
Since also $A$ is self-adjoint operator in the weighted space $l^2((1+\Delta t\lambda)\pi)$,  we can express $u^0$ using 
\begin{equation}
u^0 - u^* = \sum_{\ell=2}^{MN} c_\ell u_\ell , \quad u_\ell  \text{ is the eigenfunction corresponding to } \mu_\ell.
\end{equation}
 Therefore, we have
\begin{equation}
u^{k}-u^* = A (u^0 - u^*) = \sum_{\ell=2}^{MN} c_\ell \mu_\ell^k u_\ell,
\end{equation}
which concludes
\begin{equation}
\left\| u^{k}-1\right\|_{l^\8} \leq c  |\mu_2|^k \quad \text{ with } |\mu_2|<1.
\end{equation}

\end{proof}

\subsection{Thresholding for sharp dynamics}\label{thres1}
In this section, we combine the thresholding scheme with the Fokker-Planck dynamics to generate the inbetween motions with sharp interface, i.e., the density is described by linear combinations of two characteristic functions. In the computations later, one will see  that the thresholding scheme also helps to achieve the finite time convergence to the sharp equilibrium density.

Notice the dynamics of the Fokker-Planck equation is invariant when replacing $\rho$ by $c\rho$. Therefore, the initial density shall be adjust based on the mass conservation law \eqref{i_adj}. After this initial adjustment,
assume initial data $\rho^0_{i,j}\in \{\rho^0_s, \rho^0_b\}$, which takes alternatively the value $\rho^0_s, \rho^0_b.$ Assume the equilibrium is $\pi_{i,j}\in \{\pi_s, \pi_b\} $ which takes alternatively the value $\pi_s, \pi_b.$

To combine the thresholding scheme with the Fokker-Planck dynamics, we  need to choose the threshold $\xi^k$ at each step to conserve \eqref{conser1_n} as follows:
\\Step 1. Given $\rho^k_{i,j}\in \{\pi_s, \pi_b\}$, compute the explicit Fokker-Planck scheme \eqref{tm38} to update $\tilde{\rho}^{k+1}_{i,j}\in [\pi_s,\pi_b]$ for any $i=1,2,\cdots,N$ and $j=1, 2, \cdots, M$. 
\\Step 2. Choose threshold $\xi^{k+1}$ and define
\begin{equation}
\rho^{k+1}_{i,j}:= \pi_s \chi_{\{i,j; \tilde{\rho}^{k+1}_{i,j} \leq \xi^{k+1}\}} + \pi_b \chi_{\{i,j; \tilde{\rho}^{k+1}_{i,j} > \xi^{k+1}\}}
\end{equation}
such that $\rho^{k+1}$ satisfies \eqref{conser1_n}. 

In Step 2, $\xi^{k+1}$ can be found using bisection such that $$f(\xi^{k+1}):=\sum_i (1+ \lambda_{i,j} \Delta t) \rho_{i,j}^{k+1} - \sum_i (1+ \lambda_{i,j} \Delta t )\pi_{i,j}=0.$$

\section{EDDA based on point-clouds: Fokker-Plank equation on $\nn$ }\label{sec_fp_n}

Suppose $(\nn, d_\nn)$ is a $d$ dimensional smooth closed submanifold of $\mathbb{R}^3$.  Assume the end image on $\nn$ is described by a equilibrium density $\rho_\8(\mx): \nn \to \bR.$
Then the  Fokker-Planck equation is given by
\begin{equation}\label{FPN}
\pt_t \rho = \divn \cdot \bbs{\rho_\8 \nabla_\nn \bbs{\frac{\rho}{\rho_\8}}},
\end{equation}
where $\nabla_\nn:= \sum_{i=1}^d \tau^{\nn}_i \nabla_{\tau^{\nn}_i}$ is surface gradient, $\nabla_{\tau^{\nn}_i}=\tau^{\nn}_i \cdot \nabla$ is the tangential derivative in the direction of $\tau^{\nn}_i$ and $\divn$ is the surface divergence defined as $\divn \xi = \sum_{i=1}^d \tau_i^\nn \cdot \nabla_{\tau_i^\nn} \xi.$
This can be recast as the relative entropy formulation
\begin{equation}
\pt_t \rho = \divn \cdot \bbs{ \rho \nabla_\nn \ln\frac{\rho}{\rho_\8}} .
\end{equation}

\subsection{Construction of Voronoi tessellation and the upwind scheme  on manifold $\nn$}\label{sec_5.1}
In this section, we construct an upwind scheme based on Voronoi tessellation for manifold $\nn$, which automatically gives a positive-preserving upwind scheme for the Fokker-Planck \eqref{FPN}.

Suppose $(\nn, d_\nn)$ is a $d$ dimensional smooth closed submanifold of $\mathbb{R}^3$ and $d_{\nn}$ is  induced by the Euclidean metric in $\bR^3$. $Q:=\{\my_i\}_{i=1}^n $ are sampled from the equilibrium density $\pi=\rho^\nn_\8$. Define the Voronoi cell as
\begin{equation}
C_i:= \{\my\in \nn ; \ud_\nn(\my,\my_i)\leq \ud_\nn(\my,\my_j) \text{ for all }\my_j\in Q\},
\end{equation} 
with the volume $|C_i|=\hs^d(C_i)$.
Then $\nn=\cup_{i=1}^n  C_i$ is a Voronoi tessellation of manifold $\nn$. One can see each $C_i$ is star shaped. Denote the Voronoi face for cell $C_i$ as 
\begin{equation}
\Gamma_{ij}:= C_i\cap C_j,  \text{ and its area as } |\Gamma_{ij}|=\hs^{d-1}(\Gamma_{ij})
\end{equation} 
for any $j=1, \cdots, n$. If $\Gamma_{ij}= \emptyset$ or $i\neq j$ then we set $|\Gamma_{ij}|=0$.

 Let $\chi_{C_i}$ be the characteristic function such that $\chi_{C_i}=1$ for $\my\in C_i$ and $0$ otherwise.
For $i=1,\cdots, n$, 
$$\rho^{\text{approx}}(\my)=\sum_{i=1}^n \rho_i \chi_{C_i}(\my)$$
 is the piecewise constant probability distribution on $\nn$ provided $\sum_{i=1}^n \rho_i|C_i|=1$ and $\rho_i\geq 0$.  Let $\pi_i$ be the approximated equilibrium density at $\my_i$ satisfying $\sum_{i=1}^n \pi_i|C_i|=1$. If $\rho^{\text{approx}}(\my)=\sum_i \rho_i \chi_{C_i}(\my)$ is an approximation of density $\rho_\nn(\my)$, then $\rho_i$ is an approximation of the density $\rho_\nn(\my_i)$.  

Define the associated adjacent  grids as 
\begin{equation}
VF(i):=\{j; ~\Gamma_{ij}\neq \emptyset\}.
\end{equation}
Then using the finite volume method and the divergence theorem on manifold, we have
\begin{equation}
\frac{\ud}{\ud t } \rho_i |C_i|=\frac{\ud }{\ud t} \int_{C_i} \rho^{\text{approx}} \hs^d(C_i) = \sum_{j\in VF(i) } \int_{\Gamma_{ij}} \pi \mathbf n \cdot \nabla_\nn\left(\frac{\rho^{\text{approx}}}{\pi} \right)  \hs^{d-1}(\Gamma_{ij}),
\end{equation}
where $\mathbf{n}$ is the unit outward normal vector field on $\partial C_i$.
Based on this, we introduce the following upwind scheme. For $i=1, \cdots, n$,
\begin{equation}\label{mp}
\frac{\ud}{\ud t}\rho_i |C_i|= \frac12 \sum_{j\in VF(i)} \frac{\pi_i+ \pi_j}{|y_i-y_j|} |\Gamma_{ij}|\left( \frac{\rho_j}{\pi_j}- \frac{\rho_i}{\pi_i} \right).
\end{equation}

We now interpret the upwind scheme as
 the forward equation for a Markov process with 
 transition probability $P_{ij}$ (from $j$ to $i$) and jump rate $\lambda_j$
\begin{equation}\label{mp1}
\frac{\ud}{\ud t}\rho_i |C_i| = \sum_{j\in VF(i)} \lambda_j P_{ij} \rho_j |C_j| - \lambda_i \rho_i |C_i|,\quad i=1, 2, \cdots, n;
\end{equation}
where 
\begin{equation}\label{def59}
\begin{aligned}
\lambda_i := \frac1{2|C_i|\pi_i}\sum_{j\in VF(i)} \frac{\pi_i+ \pi_j}{|y_i-y_j|}|\Gamma_{ij}|, \quad i=1, 2, \cdots, n; \\
 P_{ij}:=\frac{1}{\lambda_j}\frac{\pi_i+ \pi_j}{2\pi_j |C_j|}\frac{|\Gamma_{ij}|}{|y_i-y_j|}, \quad j\in VF(i); \quad P_{ij}=0, \quad j\notin VF(i).
 \end{aligned}
\end{equation}
One can see
it 
satisfies $\sum_i P_{ij}=1$ and the detailed balance property
\begin{equation}\label{db}
 P_{ij} \lambda_j\pi_j |C_j| =  P_{ji} \lambda_i\pi_i |C_i|.
\end{equation}
We refer to \cite{GLW20} for the ergodicity of this Markov process.

In practice, 
instead of the $|C_i|, \Gamma_{ij}$
in \eqref{mp1}, one shall use the approximated coefficients $\tilde{C}_i$ and $\tilde{\Gamma}_{ij}$ because we do not know the exact metric information of the manifold based only on point clouds. We omit the algorithm of finding the approximated $\tilde{C}_i$ and $\tilde{\Gamma}_{ij}$ and refer to \cite{GLW20}.

\subsection{Unconditional stable explicit time stepping and exponential convergence}
 Now we propose an unconditionally stable explicit time discretization for the upwind scheme \eqref{mp}, which enjoys several good properties as the scheme \eqref{explicit_scheme}, such as maximal principle, mass conservation law and exponential convergence.

Let $\rho^{k}_i $ be the discrete density at discrete time $k\Delta t$. To achieve both stability and efficiency, we introduce the following unconditional stable explicit scheme 
\begin{equation}\label{559semi}
\frac{\rho_i^{k+1}}{\pi_i} = \frac{\rho^k_i}{\pi_i}-\lambda \Delta t \frac{\rho^{k+1}_i}{\pi_i}  +  \Delta t\sum_{j\in VF(i)} \lambda_i P_{ji}   \frac{\rho^{k}_j }{\pi_j}, \quad i=1, 2, \cdots, n
\end{equation}
which is
\begin{equation}\label{semi_1}
\frac{\rho_i^{k+1}}{\pi_i} =\frac{\rho^k_i}{\pi_i} +  \frac{ \lambda_i  \Delta t }{1+\lambda_i \Delta t} \left(  \sum_{j\in VF(i)} P_{ji}   \frac{\rho^{k}_j }{\pi_j} -\frac{\rho^k_i}{\pi_i} \right).
\end{equation}
For $u_i^{k+1}:=\frac{\rho_i^{k+1}}{\pi_i}$,  the matrix formulation of \eqref{semi_1} is
\begin{equation}\label{matrix_semi}
u^{k+1} = (I+\Delta t \hat{B}) u^k,
\end{equation}
where 
\begin{equation}\label{BBn}
\hat{B}:= \{\hat{b}_{ij}\}= \left\{\begin{array}{cc}
-\frac{ {\lambda}_i  \ }{1+{\lambda}_i \Delta t}  , \quad &j=i;\\
\frac{ {\lambda}_i   }{1+{\lambda}_i \Delta t} {P}_{ji}, \quad &j\neq i.,
\end{array}\right. \quad \text{ with } \sum_j \hat{b}_{ij}=0. 
\end{equation}

We give the following proposition for several properties of scheme \eqref{559semi}. The proof of this proposition is similar to Proposition \ref{prop1} so we omit it.
\begin{prop}
Let $\Delta t$ be the time step and consider the explicit scheme \eqref{559semi}. Assume the initial data satisfies 
\begin{equation}\label{adjust}
\sum_i (1+ \lambda_i \Delta t) \rho_i^0 |{C}_i| = \sum_i (1+ \lambda_i \Delta t )\pi_i |{C}_i|.
\end{equation}
Then
we have
\begin{enumerate}[(i)]
\item the conversational law for $g_i^{k+1}:= (1 + \Delta t {\lambda}_i ) \rho_i^{k+1}|{C}_i|$, i.e.
\begin{equation}\label{conser1}
\sum_i (1+ \lambda_i \Delta t) \rho_i^{k+1} |{C}_i| = \sum_i (1+ \lambda_i \Delta t )\rho_i^{k} |{C}_i|;
\end{equation}
\item  the unconditional maximal principle for $\frac{\rho_i}{\pi_i}$
\begin{equation}\label{maxP}
\max_j \frac{\rho^{k+1}_j }{\pi_j}\leq \max_j \frac{\rho^{k}_j }{\pi_j}.
\end{equation}
\item the $\ell^\8$ contraction
\begin{equation}\label{l8semi}
\max_i \left| \frac{\rho^{k+1}_i}{\pi_i}-1\right| \leq \max_i \left| \frac{\rho^{k}_i}{\pi_i}-1\right| ;
\end{equation}
\item the exponential convergence
\begin{equation}\label{gap_error}
\left\| \frac{\rho^{k}_i}{\pi_i}-1\right\|_{\ell^\8} \leq c  |\mu_2|^k, \quad |\mu_2|<1,
\end{equation}
where $\mu_2$ is the second eigenvalue of $I+\Delta t \hat{B}$ (in terms of magnitude), i.e. $\mu_2=1-\text{gap}_{\hat{B}}\Delta t$ and $\text{gap}_{\hat{B}}$ is the spectral gap of $\hat{B}$. 
\end{enumerate}
\end{prop}

\subsection{Thresholding for sharp dynamics}\label{thres2}
Assume the initial density is adjusted based on the mass conservation law \eqref{adjust}.  We now give the sharp dynamics by combining the Fokker-Planck equation on manifold with the thresholding scheme.

Assume initial data $\rho^0_i\in \{\rho^0_s, \rho^0_b\}$, which takes alternatively the value $\rho^0_s, \rho^0_b.$ Assume the equilibrium is $\pi_i\in \{\pi_s, \pi_b\} $ which takes alternatively the value $\pi_s, \pi_b.$ 

Similar to Section \ref{thres1}, we choose the threshold $\xi^k$ at each step to conserve \eqref{adjust} as follows.
\\Step 1. Given $\rho^k_i\in \{\pi_s, \pi_b\}$, compute the explicit scheme \eqref{559semi} to update $\tilde{\rho}^{k+1}_i\in [\pi_s,\pi_b]$ for any $i=1,2,\cdots, n$. 
\\Step 2. Choose threshold $\xi^{k+1}$ and define
\begin{equation}\label{thres}
\rho^{k+1}_i:= \pi_s \chi_{\{i; {\rho}^{k+1}_i \leq \xi^{k+1}\}} + \pi_b \chi_{\{i; {\rho}^{k+1}_i > \xi^{k+1}\}}, \quad i=1,2,\cdots, n
\end{equation}
such that $\rho^{k+1}$ satisfies \eqref{adjust}. Here $\xi^{k+1}$ can be found using bisection such that 
\begin{equation}\label{tm_bis}
f(\xi^{k+1}):=\sum_i (1+ \lambda_i \Delta t) \rho_i^{k+1} |{C}_i| - \sum_i (1+ \lambda_i \Delta t )\pi_i |{C}_i|=0.
\end{equation}

\section{Computations}\label{sec_simu}
In this section, three numerical examples are carried out to examine the capability and efficiency of the equilibrium-driven deformation algorithm (EDDA),  which are the RGB colored facial aging transformation, the pneumonia of COVID-19  invading and fading away on CT  scan images and the continental evolution process. 

\subsection{Example I: RGB colored facial aging transformation.}

In this example, we have two images with the same size in the RGB color model showing a lady’s face at two different age, and employ the model to simulate the transformation from one image (initial) to another image (equilibrium), which will illustrate the facial aging process with time. The strategy is to define each image as three matrices, each matrix containing the value of a color mode (R or G or B). Then the transformation between the two images is completed by applying the inbetweening auto-animation three times based on Fokker-Planck dynamics \eqref{FPo}.

The two images are extracted from \cite{web} and are both $355$ pixels in width and $575$ pixels in height, which means a total of $204125$ pixel points in each image. The initial image data is first adjusted to meet the mass conservation law \eqref{i_adj}. Time step  $\Delta t$ is set to $0.01$ and the total number of iterations is set to $10000$ thus the final iteration time $T=100$. The horizontal resolution $\Delta x$ and $\Delta y$ are both $10^{-4}$. We use the unconditional stable explicit time stepping scheme \eqref{explicit_scheme} and the no-flux boundary condition \eqref{nbc} to  the  Fokker-Planck equation \eqref{FPo} in domain $\Omega$.

The relative root mean square errors
\eqref{gap_error_n}
 for the three color-modes are illustrated in Fig. \ref{fig1_error3D} separately in semiology plot. Except for the different descend rates for the three colors, all simulated errors have the exponential convergence rates, which is consistent with the analysis in Proposition \ref{prop1}.

\begin{figure}
 \includegraphics[scale=0.5]{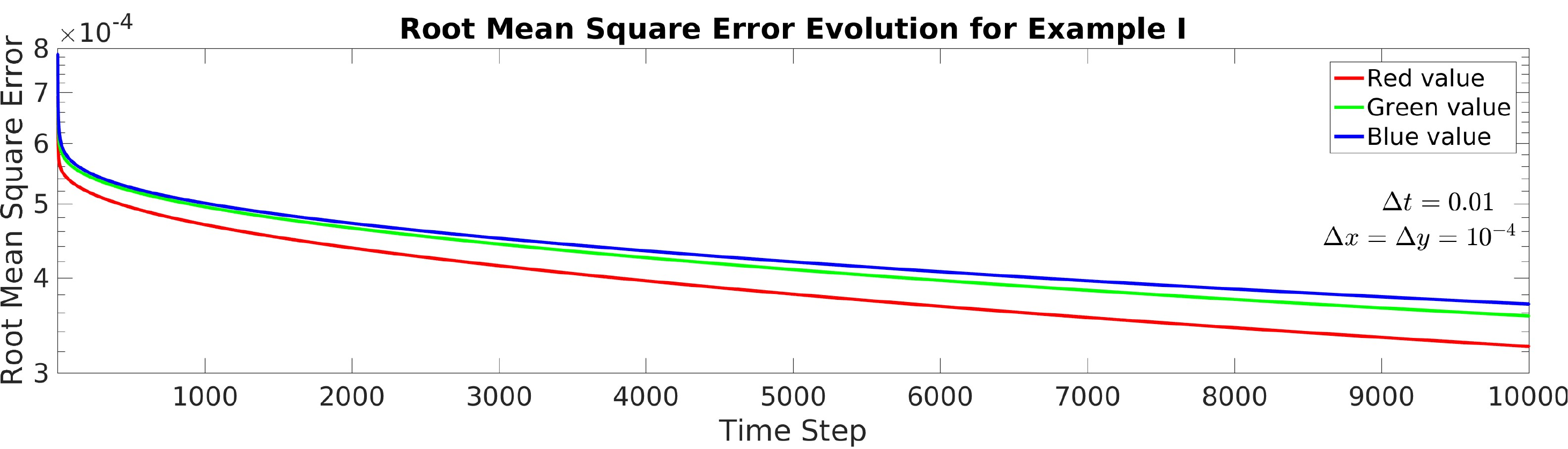} 
\caption{The semilog plot of temporal evolution of relative root mean square errors for the RGB facial aging transformation  with parameters $\Delta t =0.01,  T=100$ and $\Delta x =\Delta y=10^{-4}$. The red, green and blue lines represent the relative errors of the corresponding color modes. }\label{fig1_error3D}
\end{figure}
\begin{figure}
\includegraphics[scale=0.55]{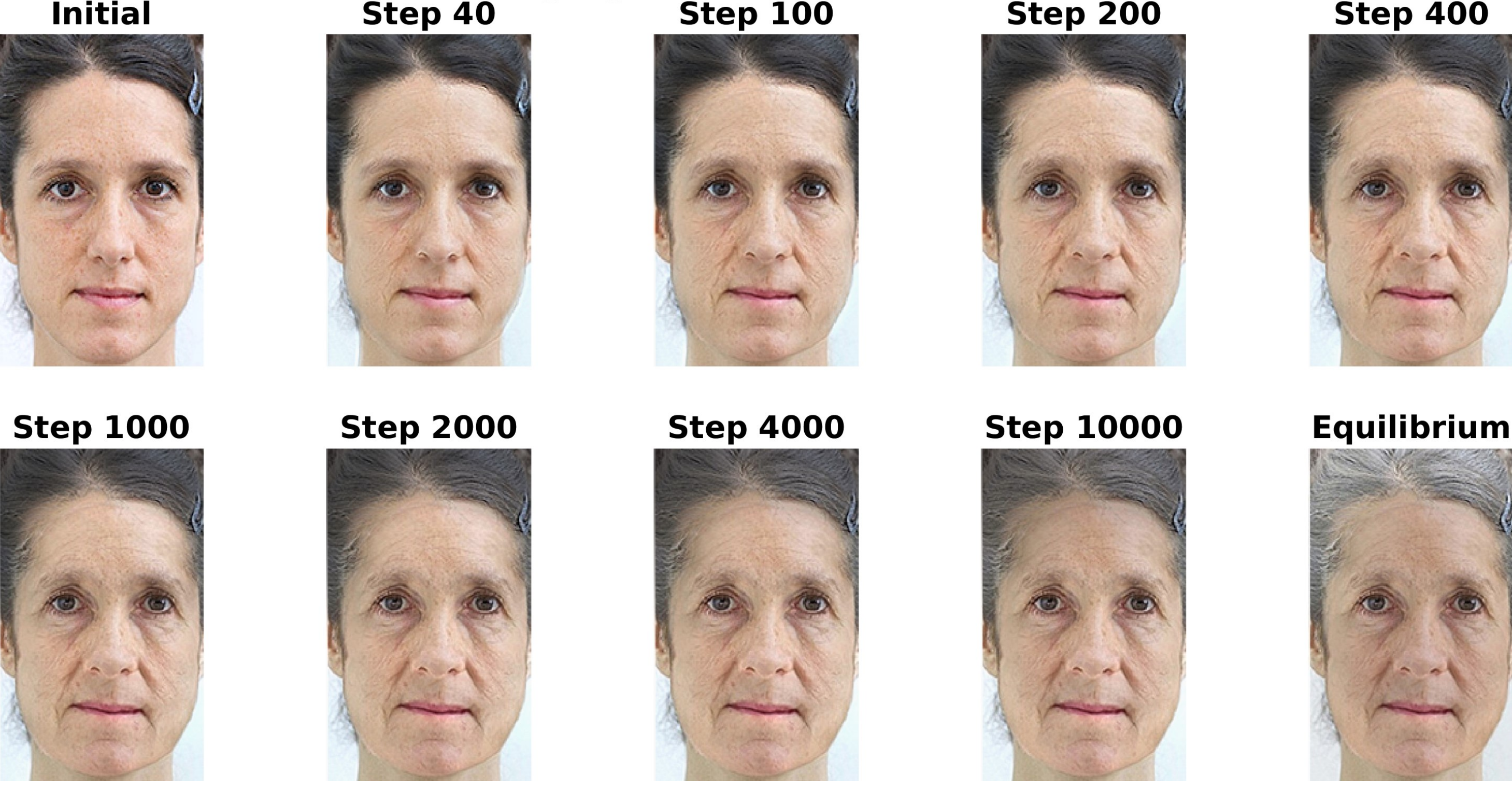} 
\caption{Facial aging transformation from initial to equilibrium with parameters $\Delta t =0.01,  T=100$ and $\Delta x =\Delta y=10^{-4}$. The updated results after time step $40, 100$, $200$, 400, $1000, 2000$, $4000$, $10000 $ are shown and compared to the initial and equilibrium images. }\label{fig2_3D}
\end{figure}

In order to see the transformation process between the two images, the images after iteration step $40, 100$, $200$, 400, $1000, 2000$, $4000$ and $10000 $ are shown and compared with the initial and the equilibrium images in Fig. \ref{fig2_3D}. The transformation process between two images are fast in the beginning (e.g. before step 200) and relatively slow after then. The transformation process in Fig. \ref{fig2_3D} clearly reveals the potential changes in different parts of the lady’s face and hair with time. After 10000 steps of iterations, the updated image is nearly the same with the equilibrium except for the hair color.

\subsection{Example II: COVID-19 pneumonia invading and fading away on CT scan images}
In this section we will focus on an example based on the COVID-19 pneumonia invading and fading away process in a patient’s lung reflected on CT scan images and try to show the possible COVID-19 pneumonia growth dynamics with time before and after the treatment. In order to fulfill the task, two parts of simulations are presented. In the first part, two CT scan images  taken on a patient’s lungs in the beginning (January 23th) and severe state (February 2nd) of the disease \cite{ZL216} are selected to be the initial and equilibrium state, respectively; see Fig. \ref{fig4a} (left). In the second part, two scan images at the severe state (February 2nd) and after a few-days’ treatment (February 9th) are selected to be the initial and equilibrium state, respectively; see Fig. \ref{fig4a} (right). Each CT scan image can be represented by a gray scale image matrix thus the same method in Example I can be applied. The CT scan images are all cropped to 461 pixels in width and 370 pixels in height, which means a total of 170570 pixel points in each image. The time step $\Delta t$ is $0.01$ and the total number of iterations is  $6000$ thus the final iteration time $T=60$. The resolutions are $\Delta x =\Delta y=10^{-4}$.
\begin{figure}
\includegraphics[scale=0.4]{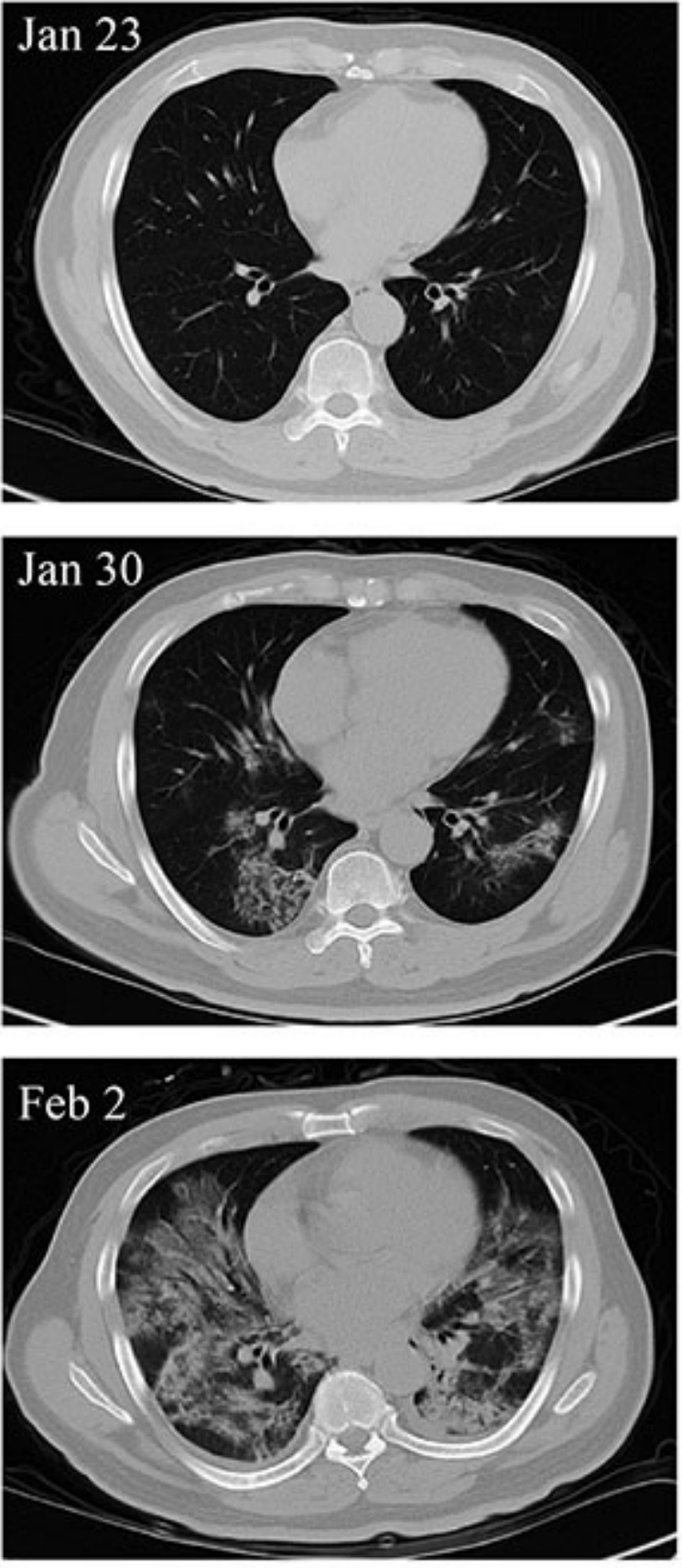}
\hspace{1cm}
 \includegraphics[scale=0.4]{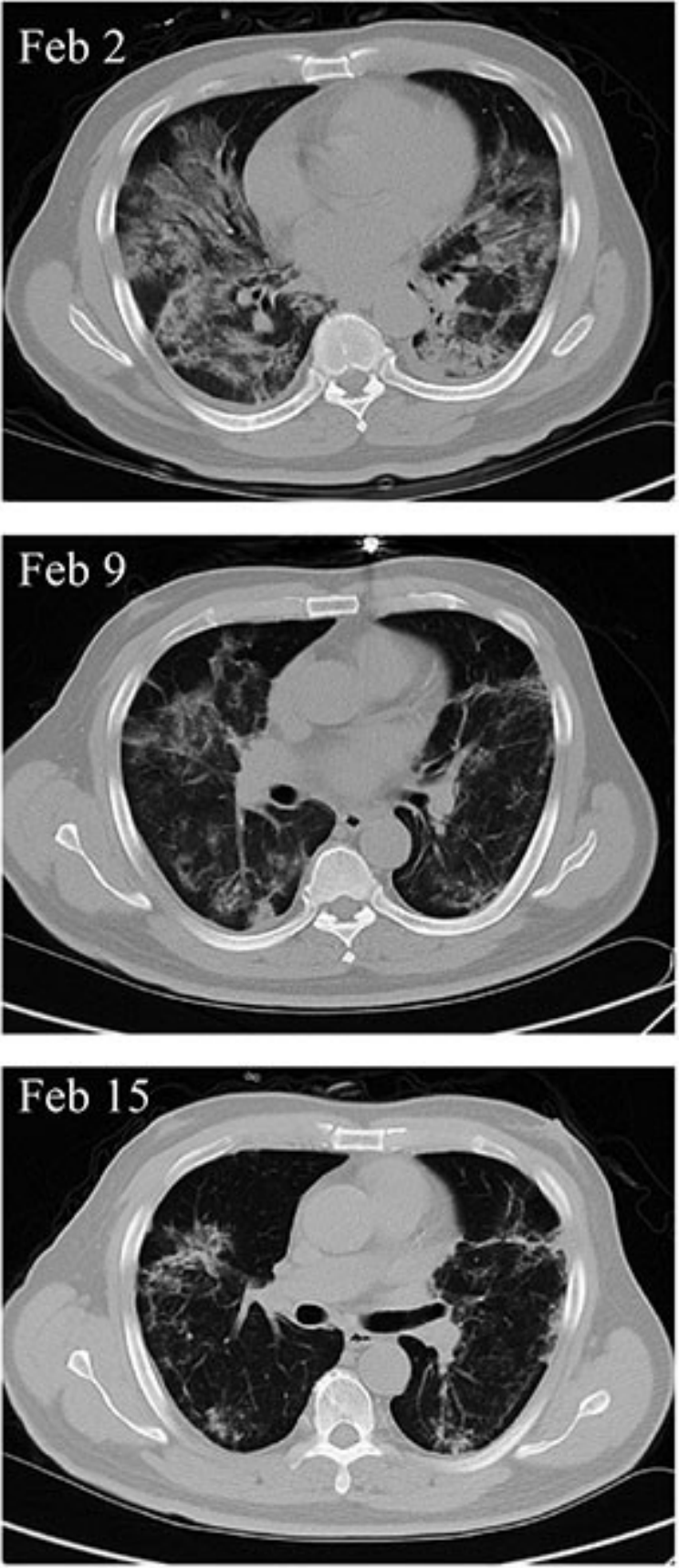} 
\caption{Chest CT images of the critically severe COVID-19 patient \cite{ZL216}. The left column of figures illustrates the evading of pneumonia from January 23th to February 2nd and the right column illustrates the fading away of pneumonia from February 2nd to February 15th after treatment.}\label{fig4a}
\end{figure}

\begin{figure}
\includegraphics[scale=0.5]{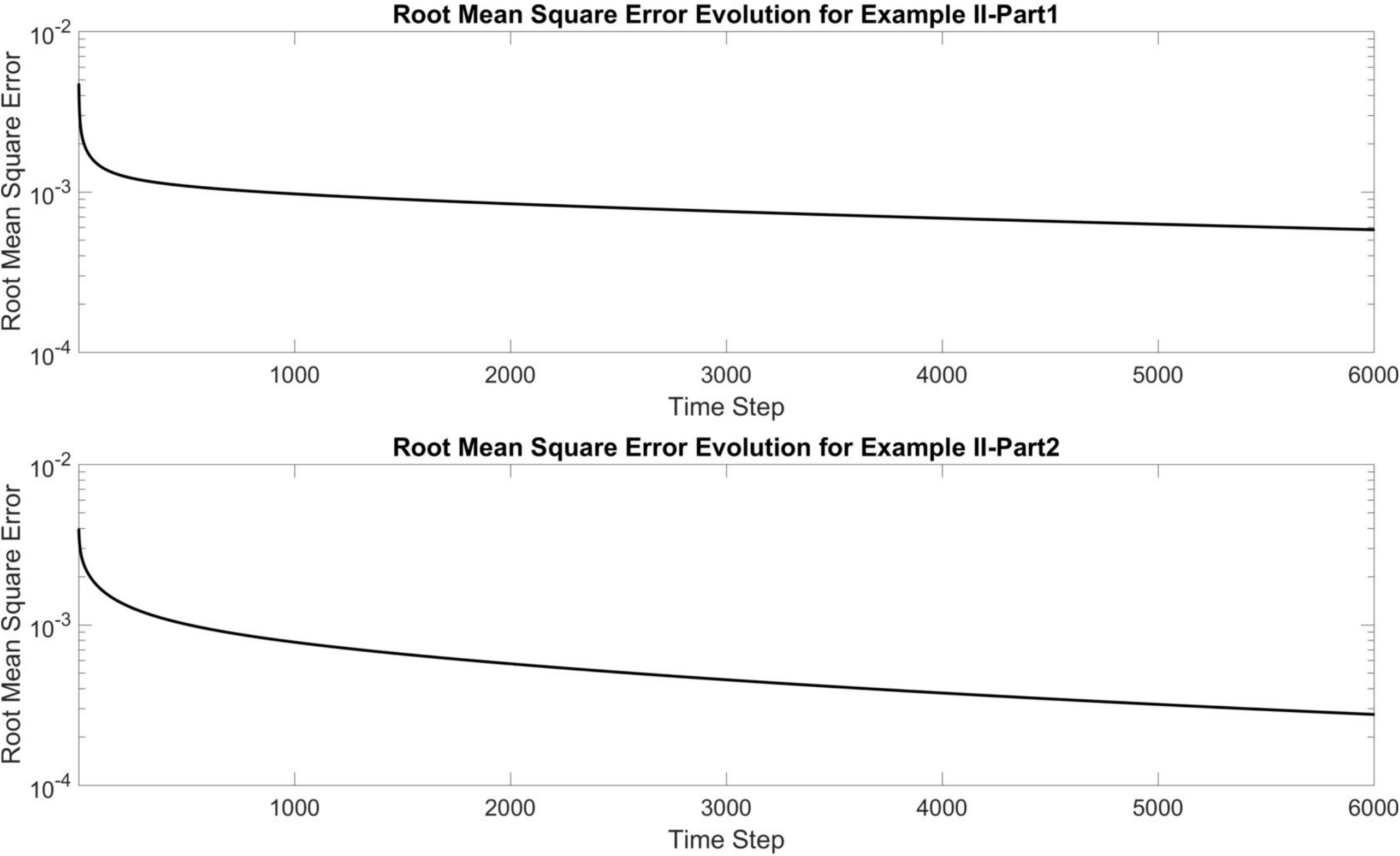} 
\caption{The semilog plot of temporal evolution of the relative root mean square errors for COVID-19 pneumonia invading and fading away on CT scan images with the parameters $\Delta t =0.01,  T=100$ and $\Delta x =\Delta y=10^{-4}$. (up) The error evolution for the pneumonia invading process simulation. (down) The error evolution for the pneumonia fading away process simulation.}\label{fig3_error2D} 
\end{figure}

After 6000 iterations, the relative root mean square errors from two parts of simulations both decrease with an exponential rate, as is shown in Fig. \ref{fig3_error2D}. Moreover, the image evolution after step 20, 50, 100, $200$, $500$, 1000, 5000, 10000 (see Fig. \ref{fig4_e2D}) clearly demonstrate the pneumonia invading process into the patient’s lungs caused by COVID-19 in a few days (upper group of figures in  Fig. \ref{fig4_e2D}) and the pneumonia fading away from the lungs after a stem cell treatment is applied to the patient (lower group of figures in Fig. \ref{fig4_e2D}), indicating a potential success of this treatment \cite{ZL216}.  We can further compare the evolution process with the real CT scan images taken on January 30th   (see Fig. \ref{fig4a}) and find satisfactory agreements, which indicates promising applications in this field.
\begin{figure}
\includegraphics[scale=0.68]{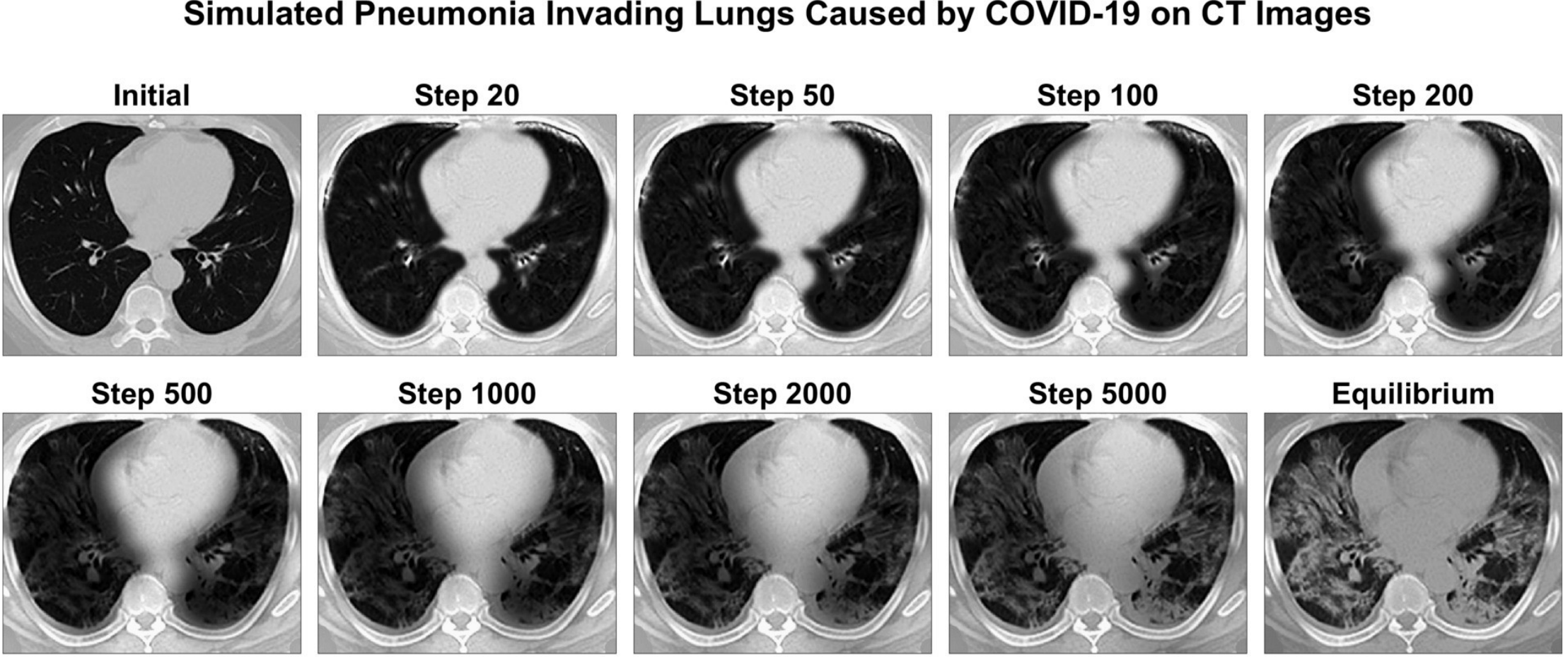} 
\includegraphics[scale=0.68]{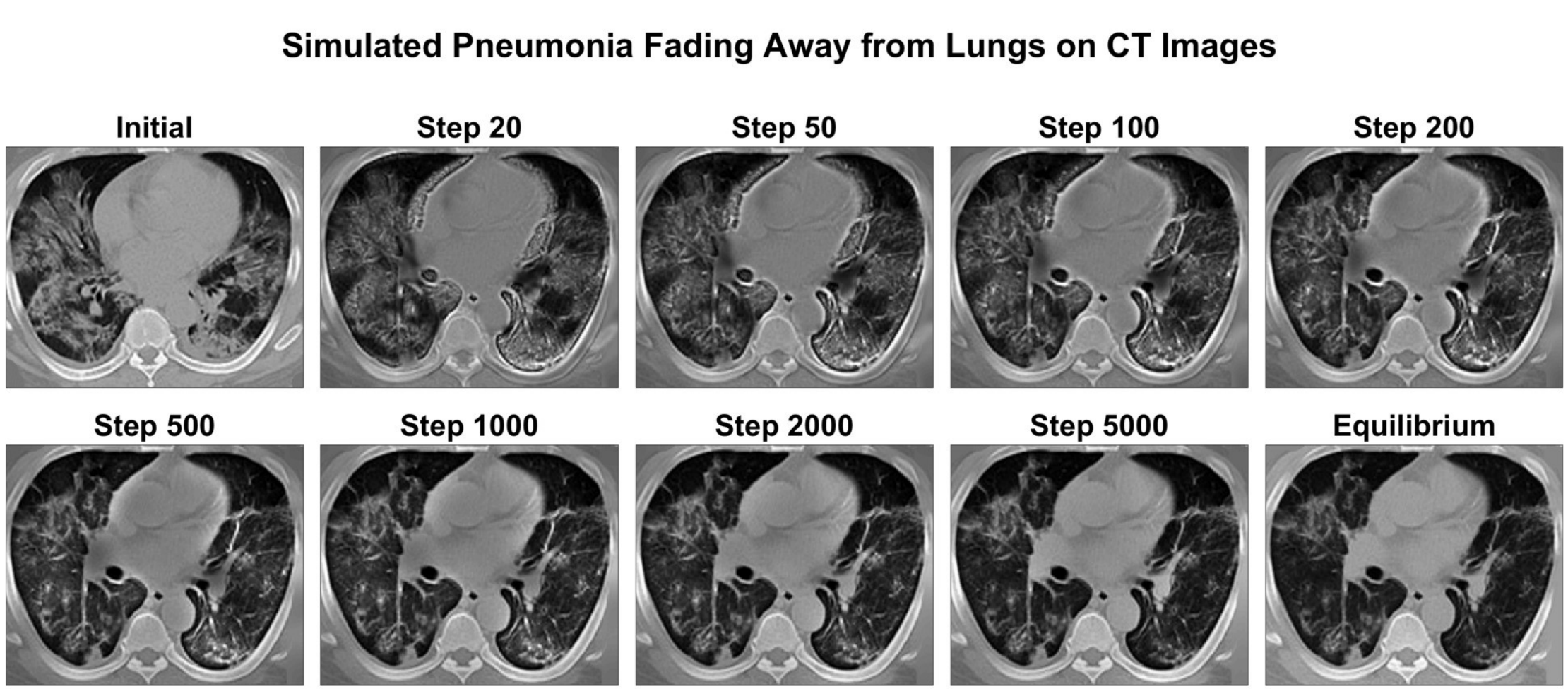} 
\caption{The simulated COVID-19 pneumonia invading and fading away process on CT scan images with the parameters $\Delta t =0.01,  T=100$ and $\Delta x =\Delta y=10^{-4}$. Results after the step 20, 50, 100, 200, 500, 1000, 2000, 5000 are illustrated and compared to the initial and equilibrium scan images. The white part inside the lungs shown on images indicates the evidence of pneumonia. (up) The pneumonia invades into the patient’s lungs caused by COVID-19. (down) The pneumonia fades away from the lungs after a stem cell treatment is applied to the patient.}\label{fig4_e2D}
\end{figure}

\subsection{Example III: Continental evolution process with thresholding for sharp dynamics}
In this section, we try to reveal the evolution process of continentals in the world from Pangaea supercontinent (250 million years ago) to the current globe. In order to clearly distinguish the sharp dynamics evolving the continentals and the oceans, the thresholding scheme \eqref{thres} and the Fokker-Planck dynamics \eqref{559semi} are combined to generate the inbetweening motions with the above sharp interfaces. The numerical experiment is carried out as follows.
\smallskip
\\
Step (I). A group of points is selected on a unit sphere to be the dataset points. With the Centroidal Voronoi Tessellation (CVT) method on the unit sphere \cite{renka1997algorithm, Du_Gunzburger_Ju_2003}, the Voronoi cells on the unit sphere are generated and the locations of dataset points are adjusted accordingly to ensure the uniformity of these cells. Thus, the distributions of continentals and oceans derived from Pangaea period and current globe’s topography are described by two values (i.e. $\pi_s$ and $\pi_b$) on the Voronoi cells and are set to be the initial and equilibrium states, respectively.  The Voronoi cell area $C_i$, $i=1, \cdots, n,$ with   the total number of dataset points $n$, is computed and the Voronoi face $\Gamma_{ij}$ is determined by the geodesic length of the neighboring arc between cell $i$ and $j$. 
\\
Step (II). Update the density at each point using the explicit scheme \eqref{559semi}  linear Fokker-Planck equation.\\
Step (III). After several linear iteration steps, the threshold is selected following the steps in Section \ref{thres2} and the thresholding scheme is applied to update the data. 

The computations for Step (II) and Step (III) will be looped until reaches the total iteration steps. Besides, a simulation case which only evolves the linear Fokker-Planck equation is carried out as the comparison. 

\begin{figure}
\includegraphics[scale=0.25]{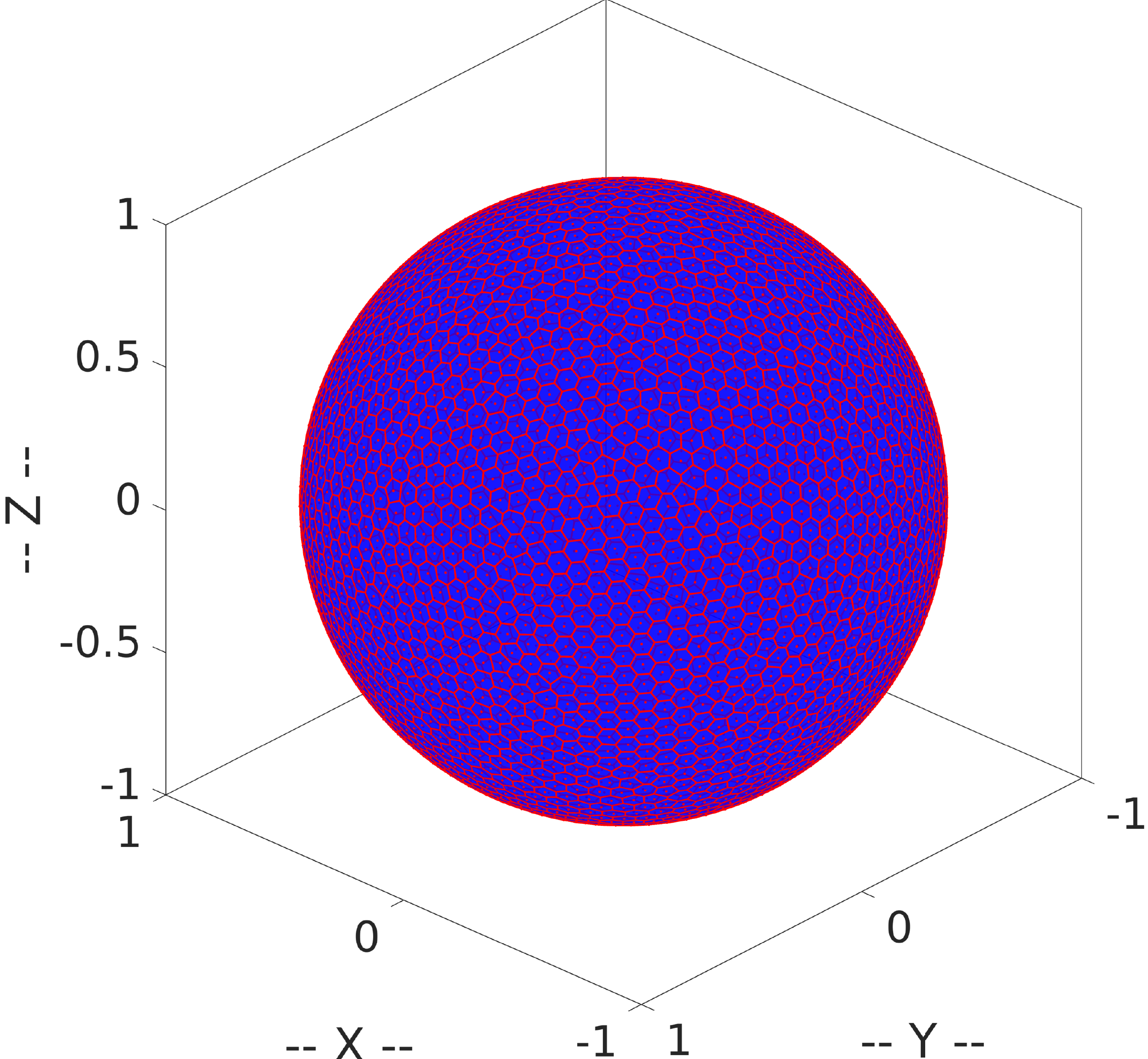} 
\caption{The unit sphere and the Voronoi cells on it. There are totally 3000 cells on the sphere. The red dots indicate locations of the point clouds on the sphere. The polygons with red edges are the Voronoi polygons generated with CVT algorithm.}\label{fig5_cell}
\end{figure}
For example III, we select a total of 3000 dataset points and generate the Voronoi cells on the unit sphere via the CVT approximation algorithm; see Fig. \ref{fig5_cell} The standard deviation for all the cell areas is $3.2\times 10^{-4}$, which means the nearly uniform distributions of data points on the sphere. The values at continental cells and the ocean cells are set to $0.9$ and $0.1$, respectively. The time step  $\Delta t$ is set to $0.05$ and the total number of linear iterations before the $(k+1)$-th thresholding adjustment is set to $2k$, $k=1, \cdots, N_t$, where $N_t=50$ is the times of the thresholding adjustments. The threshold $\xi_k$ is determined by bisection method such that \eqref{tm_bis} is satisfied. Here, the bisection domain limitation criterion is set to $10^{-6}$. The total number of iterations for the comparison simulation is set to be $10000$. 

Fig. \ref{fig6_error_c} shows the temporal variations of the relative root mean square errors for the numerical example in the first 1200 time steps. The error from the thresholding method generally have a descend trend although with some abrupt increase due to the thresholding adjustments. The error decreases to nearly zero (less than the machine accuracy) after the $30$th thresholding adjustment (a total of 960 time steps), which indicates the data is updated to the equilibrium. As a comparison, the error of the simulation via the linear method (red line in Fig. \ref{fig6_error_c}), which leads an exponential convergence rate, is smaller than that from thresholding method before the 960th time step (black circle in Fig. \ref{fig6_error_c}) and is larger after then. In order to further compare the efficiency of the two methods, we calculated the total time steps needed for the error to reach the criterions and listed them in Table \ref{table1}. The comparisons clearly reveal the efficiency of the thresholding adjustment in the application of the sharp dynamics, especially when the criterion is small. 
\begin{figure}
\includegraphics[scale=0.38]{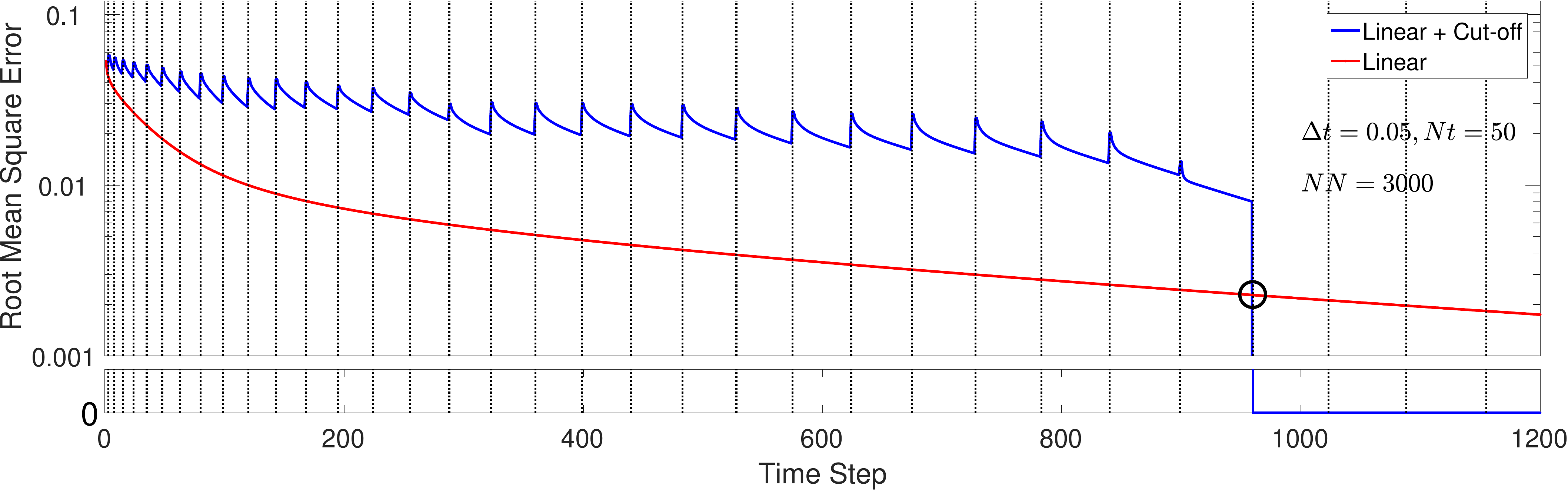} 
\caption{The semilog plot of temporal evolution of relative root mean square errors for the continental evolution process with thresholding for sharp dynamics with $\Delta t = 0.05$ and the total number of the thresholding adjustments is $N_t=50$. The linear iterations before the $(k+1)$th thresholding adjustment is  $2k$. The red line indicates the error of simulations with only the linear Fokker-Planck algorithm while the blue line is the error of the linear algorithm combined with the thresholding scheme. The black dotted lines indicate the time steps when the thresholding adjustments are applied.}\label{fig6_error_c}
\end{figure}
\begin{table}
\begin{tabular}{|c|c|c|c|c|}
\hline 
Criterion & $10^{-3}$ & $10^{-4}$ & $10^{-5}$ & $10^{-6}$ \\ 
\hline 
Steps (Linear) & 1737 & 4026 & 6304 & 8581 \\ 
\hline 
Steps (Threshold) & 960 & 960 & 960 & 960 \\ 
\hline 
\end{tabular}
\vspace{0.2cm}
\caption{Comparison of the needed steps to reach the criterions between linear method and threshold method.}\label{table1}
\end{table} 

	The continental evolution on the sphere after the $2$th, $5$th, $10$th, $30$th thresholding adjustment is illustrated and compared with the initial and equilibrium states in Fig. \ref{fig7_e_cut}. After several steps of thresholding adjustment, the sharp shapes of continentals quickly move from the initial Pangaea supercontinent towards the equilibrium state of current continentals. The distributions of continentals and oceans reaches the equilibrium state after the $28$th thresholding adjustment, exactly the same as the current distributions. Although the evolution of the continental movements is simulated with the data-driven model, some potential dynamics of continental drifting such as the Antarctic formation can be noticed in the evolutions, which may contribute to the detailed explanation of the continental drifting theory.
\begin{figure}
\includegraphics[scale=0.55]{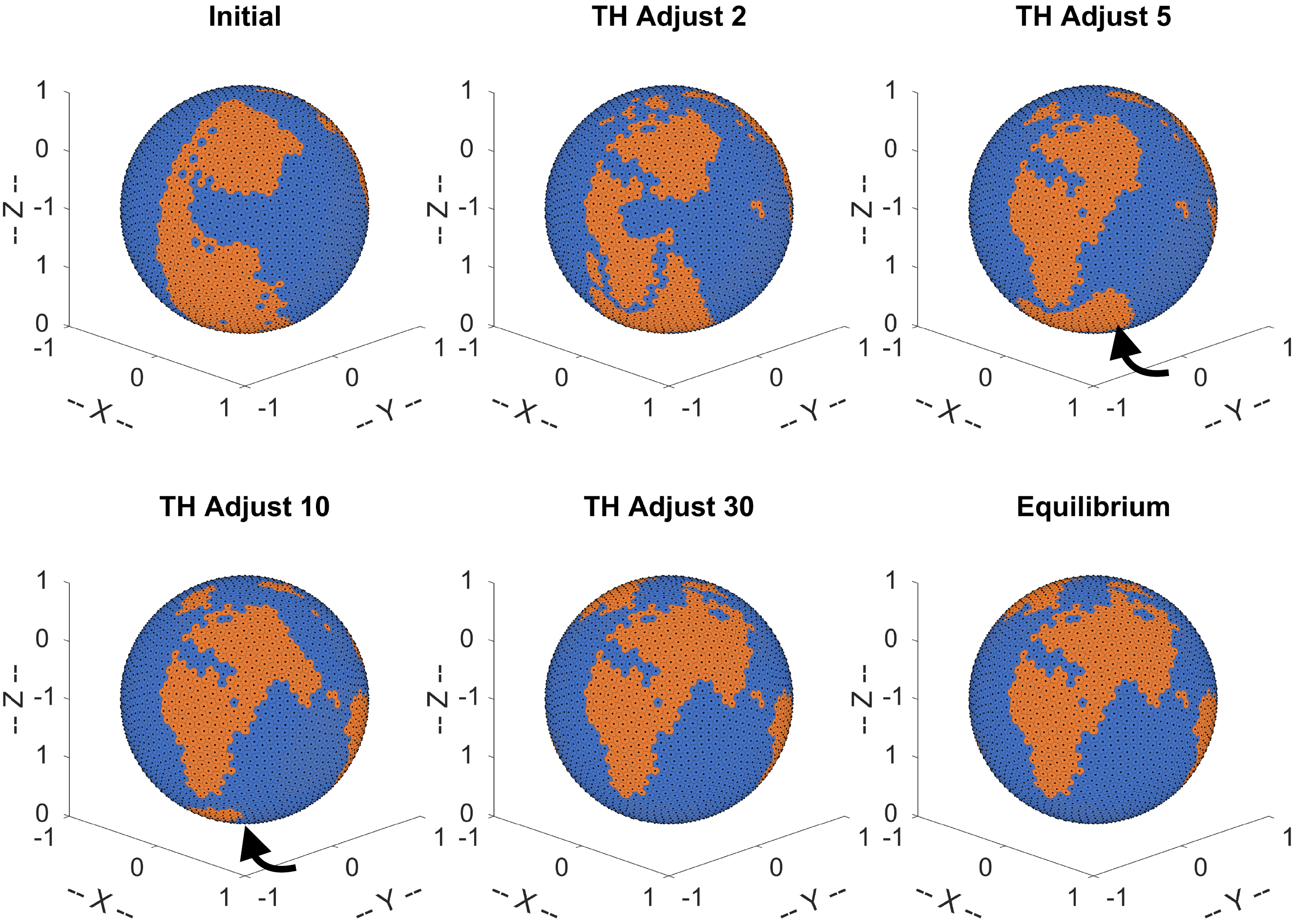} 
\caption{The evolutions of continental movements on the unit sphere with the parameter $\Delta t = 0.05$ and the total number of the thresholding adjustments is $N_t=50$. The continental evolution on the sphere after the 2th, 5th, 10th, 30th thresholding adjustment are illustrated and compared with the initial and equilibrium states. The black dots and polygons in each subplot illustrate the point clouds and the Voronoi cells, respectively. The orange and blue patch indicate the land and ocean, respectively. ‘TH’ is short for ‘thresholding step’. The formation of the Antarctic is revealed at the bottom (southern part) of the globe  (black arrow in TH 5). Note that the globes are shown in the same view angle so the Antarctic continental is out of view in the last two subplots.
}\label{fig7_e_cut}
\end{figure}

\section{Discussion}
We propose an efficient and universal equilibrium-driven deformation algorithm (EDDA) to simulate the inbetweening transformations given an initial and equilibrium. The algorithm automatically cooperates positivity, unconditional stability, mass conservation law, exponentially convergence and also  the manifold structure suggested by dataset.
Using EDDA, three challenging examples, (I) facial aging process, (II) COVID-19 invading/treatment process, and (III) continental evolution process are conducted efficiently. EDDA is shown to be a very efficient and universal method with enormous potential applications in other fields of science and  industry.

\section*{Acknowledge}
The authors would like to thank Prof. Haiyan Gao for  helpful suggestions. 
J.-G. Liu was supported in part by the National Science Foundation (NSF) under award DMS-1812573. G. Jin was supported in part by the the Natural Science Foundation of Guangdong Province under award 2019A1515011487 and the Fundamental Research Funds for the Central Universities under award 20184200031610059.

\bibliographystyle{alpha}
\bibliography{bibvis}

\end{document}